\documentclass[twoside,11pt]{amsart}

\usepackage{latexsym}
\usepackage{mathrsfs}
\usepackage[T1]{fontenc}
\usepackage{enumitem}

\usepackage{amssymb}
\usepackage{latexsym}
\usepackage{amsmath}
\usepackage{fancyhdr}
\usepackage{bbm}

\usepackage{amsaddr}

\numberwithin{equation}{section}

\newtheorem{theorem}{Theorem}[section] 
\newtheorem{lemma}[theorem]{Lemma}     
\newtheorem{corollary}[theorem]{Corollary}
\newtheorem{proposition}[theorem]{Proposition}

\theoremstyle{definition}

\newtheorem{definition}[theorem]{Definition}

\newcommand {\C}{\mathbb C}

\newcommand {\N}{\mathbb N}
\newcommand {\Z}{\mathbb Z}

\newcommand {\B}{\mathcal B}
\newcommand{\rad}{{\rm rad\,}}

\newcommand{\eps}{\varepsilon}

\title{The Radical of the Bidual of a Beurling Algebra}

\author[J.\ T.\  White]{Jared T.\ White}
\address{
Department of  Mathematics and Statistics, Lancaster University, Lancaster LA1 4YF, United Kingdom}
\email{j.white6@lancaster.ac.uk}


\date{2017}


\begin{document}

\begin{abstract}
We prove that the bidual of a Beurling algebra on $\Z$, considered as a Banach algebra with the first Arens product, can never be semisimple. We then show that $\rad(\ell^{\, 1}(\oplus_{i=1}^\infty \Z)'')$ contains nilpotent elements of every index. Each of these results settles a question of Dales and Lau. Finally we show that there exists a weight $\omega$ on $\Z$ such that the bidual of $\ell^{\, 1}(\Z, \omega)$ contains a radical element which is not nilpotent.
\end{abstract}

\maketitle
\address

\section{Introduction}
\noindent
Let $G$ be a discrete group, and let $\omega$ be a weight on $G$. Let $\ell^{\, 1}(G, \omega)''$ denote the bidual of $\ell^{\, 1}(G, \omega)$, considered as a Banach algebra with the first Arens product `$\Box$'. We shall study the Jacobson radical $\rad(\ell^{\, 1}(G, \omega)'', \Box)$ in this article.
The focus will be on the cases where either $\omega = 1$, in which case we are in fact studying the bidual of the group algebra $\ell^{\, 1}(G)$, or where the weight is non-trivial but $G= \Z$. Our main results will be solutions to two questions posed by Dales and Lau in \cite{DL}.

The study of the radicals of the biduals of Banach algebras goes back at least to Civin and Yood's paper \cite{CY}, where it was shown that if $G$ is either a locally compact, non-discrete, abelian group, or a discrete, soluble, infinite group, then $\rad(L^{\, 1}(G)'') \neq \{ 0 \}$. Civin and Yood's results have since been extended to show that $\rad(L^{\, 1}(G)'')$ is not only non-zero, but non-separable, whenever $G$ is discrete and amenable (\cite{G63}, \cite[7.31(iii)]{P}) or non-discrete \cite{G73}. The study has not been restricted to those Banach algebras coming from abstract harmonic analysis. One particularly striking result is a theorem of Daws and Read \cite{DR} which states that, for $1<p<\infty,$ the algebra $\B(\ell^{\, p})''$ is semisimple if and only if $p=2.$

A study of $\rad(\ell^{\, 1}(G, \omega)'')$ for $G$ a discrete group and $\omega$ a weight on $G$ was undertaken by Dales and Lau in \cite{DL}. In the list of open problems at the end of their memoir the authors ask whether $\ell^{\, 1}(\Z, \omega)''$ can ever be semisimple \cite[Chapter 14, Question 6]{DL}. In Section 3 we shall prove that the answer to this question is negative:

\begin{theorem} \label{1.1}		
Let $\omega$ be a weight on $\Z$. Then $\rad(\ell^{\, 1}(\Z, \omega)'') \neq \{ 0 \}$.
\end{theorem}

A key observation of Civin and Yood (see \cite[Theorem 3.1]{CY}) is that, for an amenable group $G$, the difference of any two invariant means on $\ell^{\, \infty}(G)$ always belongs to the radical of $\ell^{\, 1}(G)'',$ and this idea is what lies behind many of the subsequent results mentioned above. Dales and Lau developed a weighted version of this argument in \cite[Theorem 8.27]{DL}, and invariant means are also at the centre of our proof of Theorem \ref{1.1}.

In each of the works \cite{CY}, \cite{G63} and \cite{G73}, whenever an element of the radical of the bidual of some group algebra is constructed it is nilpotent of index 2. This is an artifact of the  method of invariant means. Moreover, it follows from \cite[Proposition 2.16]{DL} and 
\cite[Theorem 8.11]{DL} that, for a discrete group $G$, if $\omega$ is a weight on $G$ such that 
$\ell^{\, 1}(G, \omega)$ is semisimple and Arens regular, then $\rad(\ell^{\, 1}(G, \omega)'')^{\Box 2} = \{ 0 \}$. To see that this  is a large class of examples consider \cite[Theorem 7.13]{DL} and \cite[Theorem 8.11]{DL}. In \cite[Chapter 14, Question 3]{DL}, Dales and Lau ask, amongst other things, whether or not we always have $\rad(L^{\, 1}(G)'')^{\Box 2} = \{ 0 \},$ for $G$ a locally compact group. It also seems that until now it was not known whether or not $\rad( L^{\, 1}(G, \omega)'')$ is always nilpotent, for $G$ a locally compact group and $\omega$ a weight on $G$, although there is an example of a weight on $\Z$ in \cite[Example 9.15]{DL} for which this radical cubes to zero, but has non-zero square. In Section 4 we shall answer both of these questions in the negative by proving the following:

\begin{theorem} \label{1.3}
Let $G= \oplus_{i=1}^\infty \Z$. Then $\rad(\ell^{\, 1}(G)'')$ contains nilpotent elements of every index.
\end{theorem}
\noindent
Here we understand $\oplus_{i=1}^\infty \Z$ to consist of integer sequences which are eventually zero, so that our example is a countable abelian group.

We note that by a theorem of Grabiner \cite{Gr69}, Theorem \ref{1.3} implies that $\rad(\ell^{\, 1}(\oplus_{i=1}^\infty \Z)'')$ contains non-nilpotent elements. In Section 5, we obtain a similar result on $\Z$, but this time involving a weight. 

\begin{theorem} \label{1.2}		
There exists a weight $\omega$ on $\Z$ such that $\rad(\ell^{\, 1}(\Z, \omega)'')$ contains non-nilpotent elements.
\end{theorem}
\noindent
However, we do not know whether or not this example has nilpotent elements of arbitrarily high index.

\section{Background and Notation}
\noindent
For us $\N = \{ 1, 2, \ldots \}$ and $\Z^+ = \{ 0,1,2, \ldots \}$. The group of integers is denoted by $\Z$ and 
$$\oplus_{i=1}^\infty \Z = \{(n_i) \in \Z^\N: n_i = 0 \text{ for all but finitely many } i \}.$$

Let $G$ be a discrete group with neutral element $e$. We say that a function $\omega \colon G \rightarrow [1, \infty)$ is a \textit{weight} if $\omega(st) \leq \omega(s)\omega(t) \ (s, t \in G)$ and $\omega(e) = 1$. Given a weight $\omega$ on a group $G$, we define
$$\ell^{\, 1}(G, \omega) = \left \{ f \colon G \rightarrow \C : \Vert f \Vert_\omega := \sum_{s \in G} \vert f(s) \vert \omega(s) < \infty \right \}.$$
This is a unital Banach algebra with the norm given by $\Vert \cdot \Vert_\omega$, multiplication given by convolution, and the vector space operations given pointwise. By a \textit{Beurling algebra} we mean a Banach algebra of the form $\ell^{\, 1}(G, \omega.)$ (More generally, authors refer to algebras of the form $L^1(G, \omega),$ where $G$ is a locally compact group, and $\omega$ is a continuous weight on $G$, as Beurling algebras, but we shall not consider this setting here.)
The dual space of $\ell^{\, 1}(G, \omega)$ may be identified with 
$$\ell^{\, \infty}(G, 1/\omega) = \left\{ f \colon G \rightarrow \C : \Vert f \Vert_{\infty, \omega} := \sup_{s \in G} \frac{\vert f(s) \vert}{\omega(s)} < \infty \right\}$$
via
$$\langle f, g \rangle = \sum_{s \in G} f(s) g(s) \quad (f \in \ell^{\, 1}(G, \omega), g \in \ell^{\, \infty}(G, 1/\omega)).$$
Given a group element $s\in G,$ we denote the point mass at $s$ by $\delta_s \in \ell^{\, 1}(G, \omega)$.

Given a weight $\omega$ on $\Z$ and $n \in \Z$, we sometimes write $\omega_n$ in place of $\omega(n).$ We define 
$\rho_\omega = \inf_{n \in \N} \omega(n)^{1/n};$
by \cite[Proposition A.1.26 (iii)]{D}, in fact $\rho_\omega = \lim_{n \rightarrow \infty}\omega(n)^{1/n}$.

Let $A$ be an algebra, and take $n \in \N$. We say that $a \in A$ is \textit{nilpotent of index $n$} if $a^n = 0$, but $a^{n-1} \neq 0$.
Given a left ideal $I$ of $A$ and $n \in \N,$ we write $I^n = \{ a_1 a_2 \cdots a_n : a_1, \ldots, a_n \in I \}$, and we say that $I$ is \textit{nilpotent of index $n$} if $I^n = \{ 0 \}$ but $I^{n-1} \neq \{ 0 \}$.

Now let $A$ be a unital Banach algebra. We say that $a \in A$ is \textit{quasi-nilpotent} if its spectrum is zero,
 or, equivalently, if $\lim_{n \rightarrow \infty}
 \Vert a^n \Vert^{1/n} = 0$, and we denote the set of quasi-nilpotent elements
  of $A$ by $\mathcal{Q}(A).$ 
 Every nilpotent element is also quasi-nilpotent. We define the \textit{Jacobson radical} of $A$,
denoted by $\rad(A)$, to be the largest left ideal of $A$ contained in $\mathcal{Q}(A)$, and it can be shown that 
$$\rad(A) = \{ a \in A : ba \in  \mathcal{Q}(A) \ (b \in A) \}.$$
In fact, $\rad(A)$ is a closed, two-sided ideal of $A$, and 
$$\rad(A) = \{ a \in A : ab \in  \mathcal{Q}(A) \ (b \in A) \}.$$
Many equivalent characterizations of $\rad(A)$ are available (for details see \cite[Section 1.5]{D}). 

Denote the dual space of $A$ by $A'$ and its bidual by $A''$. Arens (\cite{A1}, \cite{A2}) introduced two products on $A''$, denoted by $\Box$ and $\Diamond$, rendering it a Banach algebra, both of which have the property that they agree with the original multiplication on $A$, when $A$ is identified with its image under the canonical embedding into $A''$. These are called, respectively, the \textit{first} and \textit{second Arens product}, and are defined by
\begin{align*}
\langle \Psi \Box \Phi, \lambda \rangle &= \langle \Psi, \Phi \cdot \lambda \rangle, &
\langle \Psi \Diamond \Phi, \lambda \rangle &= \langle \Phi, \lambda \cdot \Psi \rangle, \\
\langle \Phi \cdot \lambda, a \rangle &= \langle \Phi, \lambda \cdot a \rangle, &
\langle \lambda \cdot \Psi, a \rangle &= \langle \Psi, a \cdot \lambda \rangle, \\
\langle \lambda \cdot a, b \rangle  &= \langle \lambda, ab \rangle, &
\langle a \cdot \lambda, b \rangle &= \langle \lambda, ba \rangle,
\end{align*}
for $\Phi, \Psi \in A'', \lambda \in A', a, b \in A$ (for more details see \cite[Section 2.6]{D}).
In this article, unless we specify otherwise, whenever we talk about the bidual of a Banach algebra we are implicitly considering it as an algebra with the first Arens product. The first Arens product has the property that multiplication on the right is weak-* continuous, whereas the second Arens product has this property on the left. In particular the following formulae hold, for $\Phi, \Psi$ elements of $A''$, and $(a_\alpha), (b_\beta) \subset A$ nets converging in the weak-* topology to $\Phi$ and $\Psi$ respectively:
\begin{equation*}
\Phi \Box \Psi = \lim_\alpha \lim_\beta a_\alpha b_\beta, \quad \Phi \Diamond \Psi = \lim_\beta \lim_\alpha a_\alpha b_\beta.
\end{equation*}
In these formulae the limits are again taken in the weak-* topology. 
If $\Box = \Diamond,$ we say that $A$ is \textit{Arens regular}, and if the other extreme occurs, namely that 
\begin{multline*} 
\left \{ \Phi \in A'' : \Phi \Box \Psi = \Phi \Diamond \Psi \ (\Psi \in A'') \right \} = \\
\left \{ \Phi \in A'' : \Psi \Box \Phi = \Psi \Diamond \Phi \ (\Psi \in A'') \right \} = A,
\end{multline*}
 we say that $A$ is \textit{strongly Arens irregular}.  Both of these extremes may occur for Banach algebras of the form $\ell^{\, 1}(\Z, \omega)$, as may intermediate cases (see \cite[Theorem 8.11]{DL} and \cite[Example 9.7]{DL}).

We fix some notation relating to repeated limits. Let $X$ and $Y$ be topological spaces, $I$ a directed set, and $\mathcal{U}$ a filter on $I$. Let $(x_\alpha)_{\alpha \in I}$ be a net in $X$, let $r \in \N,$ and let $f : X^{r} \rightarrow Y$ be a function. Then we define 
$$\lim_{\underline{\alpha} \rightarrow \mathcal{U}} {}^{(r)} f(x_{\alpha_1}, \ldots, x_{\alpha_r}) = \lim_{\alpha_1 \rightarrow \mathcal{U}} \cdots \lim_{\alpha_r \rightarrow \mathcal{U}} f(x_{\alpha_1}, \ldots, x_{\alpha_r}),$$
whenever the repeated limit exists. We define 
$$\limsup_{\underline{\alpha} \rightarrow \mathcal{U}} {}^{(r)} f(x_{\alpha_1}, \ldots, x_{\alpha_r})$$
 analogously. Suppose now that we have two directed sets $I$ and $J$ and two filters: $\mathcal{U}$ on $I$ and $\mathcal{V}$ on $J$. Let $(x_\alpha)_{\alpha \in I}$ and $(y_\beta)_{\beta \in J}$ be two nets in $X$, let $r \in \N,$ and let $f : X^{2r} \rightarrow Y$. Then we define
\begin{multline*}
\lim_{\underline{\alpha} \rightarrow \mathcal{U}, \, \underline{\beta} \rightarrow \mathcal{V}}   {}^{(r)} f(x_{\alpha_1}, y_{\beta_1}, \ldots, x_{\alpha_r},  y_{\beta_r})= \\
 \lim_{\alpha_1 \rightarrow \mathcal{U}} \lim_{\beta_1 \rightarrow \mathcal{V}} \cdots \lim_{\alpha_r \rightarrow \mathcal{U}} \lim_{\beta_r \rightarrow \mathcal{V}}  f(x_{\alpha_1}, y_{\beta_1}, \ldots, x_{\alpha_r},  y_{\beta_r}),
\end{multline*}
whenever the limit exists.
It is important to note that the choice of directed set in the above repeated limit alternates. In expressions of the form $\lim_{\underline{\alpha} \rightarrow \infty} {}^{(r)} f(x_{\alpha_1}, \ldots, x_{\alpha_r})$ the symbol `$\infty$' is understood to represent the Fr\'echet filter on the  directed set.

\section{Semisimplicity of $\ell^{\, 1}(\Z, \omega)''$}
\noindent
In this section, we shall prove Theorem \ref{1.1}. Throughout $\omega$ will be a weight on $\Z$, and we shall write $A_\omega = \ell^{\, 1}(\Z, \omega)$.
In an abuse of notation, we shall write $1 \in \ell^{\, \infty}(\Z)$ for the sequence which is constantly 1. Note that this is a character, known as the augmentation character, when regarded as an element of $A_\omega'$. We define 
$$I_\omega = \left \{ \Lambda \in A_\omega'': \delta_n \Box \Lambda = \Lambda \ (n \in \Z), \; \langle \Lambda, 1 \rangle = 0 \right  \}.$$
By \cite[Proposition 8.23]{DL} $I_\omega$ is an ideal of $A_\omega''$, satisfying $I_\omega^{\Box 2} = \{ 0 \}$, so that $I_\omega \subset \rad(A_\omega'')$. Our strategy will be to reduce to a setting in which we can show that $I_\omega \neq \{ 0 \}$. Our argument is an adaptation of \cite[Theorem 8.27]{DL}.

Let $\Lambda \in \ell^{\, \infty}(\Z, 1/\omega)'$. We say that $\Lambda$ is \textit{positive}, written $\Lambda \geq 0$, if $\langle \Lambda, f \rangle \geq 0$ whenever $f \geq 0 \ (f \in \ell^{\, \infty}(\Z, 1/\omega))$, and we say that $\Lambda$ is a \textit{mean} if $\Lambda \geq 0$ and $\Vert \Lambda \Vert = 1$.
We say that a mean $\Lambda \in  \ell^{\, \infty}(\Z, 1/\omega)'$ is an \textit{invariant mean} if $\delta_n \Box \Lambda = \Lambda \ (n \in \Z)$. 

\begin{lemma} \label{0}	
Let $\omega$ be a weight on $\Z$ and let $\Lambda \in \ell^{\, \infty}(\Z, 1/\omega)'$ be positive. Then $\Vert \Lambda \Vert = \langle \Lambda, \omega \rangle$.
\end{lemma}

\begin{proof}
This follows by considering the positive isometric Banach space isomorphism $T \colon \ell^{\, \infty}(\Z, 1/\omega) \rightarrow \ell^{\, \infty}(\Z)$
given by $T(f) = f/\omega \ (f \in \ell^{\, \infty}(\Z, 1/\omega))$, and then using the facts that the formula holds in the C*-algebra $\ell^{\, \infty}(\Z)$ and that $T(\omega) = 1.$
\end{proof}
 
 In what follows, given $E \subset \N$ we denote the complement of $E$ by $E^c$.
 
\begin{lemma} \label{1}
Let $\omega$ be a weight on $\Z,$ and suppose that $\rho_\omega = 1.$ Then there exist at least two distinct invariant means $\Lambda$ and $M$ on $\ell^{\, \infty}(\Z, 1/\omega)$ such that $\langle \Lambda, 1 \rangle = \langle M, 1 \rangle$.
\end{lemma}

\begin{proof}
By an easy calculation, using the fact that $\inf_{n \in \N} \omega_n^{1/n} = 1$ (see \cite[Lemma 5.3]{W}), there exists a strictly increasing sequence $(n_k)$ of integers such that $n_0 = 0, n_1 = 1$ and such that
\begin{equation} \label{eq1}	
\lim_{k \rightarrow \infty} \omega_{n_k}/(\omega_0 + \cdots + \omega_{n_k}) = 0. 
\end{equation}
By passing to a subsequence if necessary we may suppose that 
$$\lim_{k \rightarrow \infty} (n_k+1)/(\omega_0 + \cdots + \omega_{n_k})$$
exists.

Set $C_k = \omega_0 + \cdots + \omega_{n_k}$, and define $\Lambda_k = \frac{1}{C_k}(\delta_0 + \cdots + \delta_{n_k})$; we regard each $\Lambda_k$ as an element of $A_\omega''$.  
Notice that, for each fixed $i \in \N,$ we have
\begin{equation}  \label{eq3}	 
\lim_{k \rightarrow \infty}C_i/C_k = 0.
\end{equation}
We shall first show that the sequence $(\Lambda_k)$ does not converge when considered as a sequence in $A_\omega''$ with the weak-* topology. 
This will then allow us to use two different ultrafilters in such a way as to obtain distinct limits of $(\Lambda_k)$, and these limits will turn out to be our 
invariant means. To achieve this,
we shall inductively construct a function $\psi \colon \Z \rightarrow \C$ and choose non-negative integers  
$$s_1 <t_1 < s_2 < t_2  < \cdots < s_k < t_k < \cdots$$
such that
\begin{equation} \label{eq2} 
\vert \langle \Lambda_{s_j}, \psi \rangle \vert < \frac{1}{4}, \qquad  \vert \langle \Lambda_{t_j}, \psi \rangle \vert >\frac{3}{4} \quad (j \in \N)
\end{equation}
and
\begin{equation} \label{eq2a}
0 \leq \psi(i) \leq \omega_i+1 \ (i \in \Z).
\end{equation}
Since \eqref{eq2a} ensures that $\psi \in \ell^{\, \infty}(\Z, 1/\omega),$ this will indeed show that $(\Lambda_k)$ is weak-* divergent.
 We set $s_1 =0$ and $t_1 = 1$,
 and define $\psi(i) = 0 \ (i \leq 0)$
and $\psi(1) = C_1$, and observe that this ensures that \eqref{eq2} holds for $j=1$, and that \eqref{eq2a} holds for all $i \leq 1$.

Now assume inductively that we have found $s_1 <t_1 < \cdots < s_k < t_k$, and defined $\psi$ up to $n_{t_k}$ in such a way that \eqref{eq2} holds for $j=1, \ldots, k$, and such that \eqref{eq2a} holds for $i \leq n_{t_k}$.  
By \eqref{eq3}, we may choose $s_{k+1} > t_k$ such that 
$$\frac{C_{t_k}}{C_{s_{k+1}}} < \frac{1}{4} \vert \langle \Lambda_{t_k}, \psi \rangle \vert^{-1};$$
we then define $\psi(i) = 0 \ (n_{t_k} < i \leq n_{s_{k+1}}),$ and note that \eqref{eq2a} holds
 trivially for these values of i. Then
$$\vert \langle \Lambda_{s_{k+1}}, \psi \rangle \vert = \frac{C_{t_k}}{C_{s_{k+1}}} \vert \langle \Lambda_{t_k}, \psi \rangle \vert < \frac{1}{4},$$
as required.

Again using \eqref{eq3}, we may choose $t_{k+1} > s_{k+1}$ such that 
$C_{s_{k+1}}/C_{t_{k+1}} < 1/8,$
so that
$$\frac{\omega(n_{s_{k+1}}+1)+ \cdots + \omega(n_{t_{k+1}})}{C_{t_{k+1}}} > \frac{7}{8}.$$
Set $\psi(i) = \omega_i \ (n_{s_{k+1}} < i \leq n_{t_{k+1}}),$ and note
that \eqref{eq2a} continues to hold. Then
\begin{align*}
\vert \langle \Lambda_{t_{k+1}}, \psi \rangle \vert &= \left \vert \frac{\omega(n_{s_{k+1}}+1)+ \cdots + \omega(n_{t_{k+1}})}{C_{t_{k+1}}}  + 
 \frac{C_{s_{k+1}}}{C_{t_{k+1}}}\langle \Lambda_{s_{k+1}}, \psi \rangle \right \vert \\
&> \frac{7}{8} - \frac{1}{8} \cdot \frac{1}{4} > \frac{3}{4}.
\end{align*}
The induction continues.

Let $\mathcal{F}$ denote the Fr\'echet filter on $\N$. Let $\mathcal{U}$ be a free ultrafilter on $\N$, and set $\Lambda = \lim_{k \rightarrow \mathcal{U}} \Lambda_k$ (the limit being taken in the weak-* topology on $A_\omega''$). 
We have shown that the sequence $(\Lambda_k)$ is not convergent in the weak-* topology on $A_\omega''$, and so it follows that there exists a weak-* open neighbourhood $\mathcal{O}$ of $\Lambda$ such that $E := \{k \in \N: \Lambda_k \in \mathcal{O} \} \notin \mathcal{F}$. As $E \notin \mathcal{F}$, the set $E^{c}$ is infinite, so that $E^{c} \cap A \neq \emptyset \ (A \in \mathcal{F})$. Hence there exists a free ultrafilter $\mathcal{V}$ on $\N$ containing $E^c$ and $\mathcal{F}$. Let $M = \lim_{k \rightarrow \mathcal{V}} \Lambda_k$. Since $E^{c} \in \mathcal{V}$, we have $E \notin \mathcal{V}$, so that $\Lambda \neq M$.

Next we show that $\Lambda$ and $M$ are invariant means on $\ell^{\, \infty}(\Z, 1/\omega)$. Let $f \in \ell^{\, \infty}(\Z, 1/\omega)$. Then 
\begin{align*}
\vert \langle \delta_1 \Box \Lambda_k - \Lambda_k, f \rangle \vert  &= 
\frac{1}{C_k} \vert (f(1) + \cdots + f(n_k+1)) - (f(0) + \cdots + f(n_k))\vert \\
&= \frac{1}{C_k} \vert f(n_k+1) - f(0)\vert \\
&\leq \frac{1}{C_k} \Vert f \Vert (\omega(1)\omega(n_k)+\omega(0)),
\end{align*}
which, by \eqref{eq1}, tends to zero as $k \rightarrow \infty$. Hence 
$$\delta_1 \Box \Lambda - \Lambda = \lim_{k \rightarrow \mathcal{U}}(\delta_1 \Box \Lambda_k - \Lambda_k) = 0,$$
and a similar calculation shows that $\delta_{-1} \Box \Lambda = \Lambda$ as well. It follows that $\Lambda$ is invariant. That $\Lambda \geq 0$ is clear, and hence, by Lemma \ref{0}, 
$$\Vert \Lambda \Vert = \langle \Lambda, \omega \rangle  = \lim_{k \rightarrow \mathcal{U}} \langle \Lambda_k, \omega \rangle = 1.$$
 Hence $\Lambda$ is an invariant mean, as claimed. The same argument shows that $M$ is also an invariant mean.

Finally, we calculate that 
\begin{equation*}
\langle \Lambda, 1 \rangle = \lim_{k \rightarrow \mathcal{U}} \langle \Lambda_k, 1 \rangle \\
= \lim_{k \rightarrow \infty} (n_k+1)/C_k = \lim_{k \rightarrow \mathcal{V}} \langle \Lambda_k, 1 \rangle \\
= \langle M, 1 \rangle,
\end{equation*}
as required.
\end{proof}

We now prove the main result of this section.
\begin{proof}[Proof of Theorem \ref{1.1}]
Let $\rho = \rho_\omega$, and let $\gamma_n = \omega_n/ \rho^n \ (n \in \Z).$ Then $\gamma$ is a weight on $\Z$, and $T \colon (f(n)) \mapsto (\rho^n f(n))$ defines an (isometric) isomorphism of Banach algebras $A_\gamma \rightarrow A_\omega$. The weight $\gamma$ satisfies the hypothesis of Lemma \ref{1}, so that there exist distinct invariant means $\Lambda$ and $M$ on $A_\gamma''$ as in that lemma. Then $\langle \Lambda - M, 1 \rangle = 0$, so that $\Lambda - M \in I_\gamma \setminus \{ 0 \}$. Hence, by \cite[Proposition 8.23]{DL}, $\rad(A_\gamma'') \neq \{ 0 \},$ so that $\rad(A_\omega'') \neq \{ 0 \}$. 
\end{proof}

\textit{Remark.}
A trivial modification of the proof of Theorem \ref{1.1} shows that in fact we also have $\rad(\ell^{\, 1}(\Z^+, \omega)'') \neq \{ 0 \}$ for every weight $\omega$ on $\Z^+$.
\medskip

\textit{Remark.} Since $A_\omega$ is commutative, $(A_\omega'', \Diamond) = (A_\omega'', \Box)^{\rm op}$, and it follows that $\rad(A_\omega'', \Diamond) = \rad(A_\omega'', \Box)$, so that $(A_\omega'', \Diamond)$ is never semisimple either.

\section{The Radical of $\ell^{\, 1}(\oplus_{i=1}^\infty \Z)''$}

In this section we prove Theorem \ref{1.3}. In addition we observe in Corollary \ref{5.5} that there  are many non-amenable groups $G$ for which
 $\rad(\ell^{\, 1}(G)'') \neq \{ 0 \}$. Ideals of the following form will be central to both  of these arguments.

\begin{definition} \label{5.1}
Let $G$ be a group, let $\theta \colon \ell^{\, 1}(G) \rightarrow \ell^{\, 1}(G)$ be a bounded algebra homomorphism, and let $J \subset \ell^{\, 1}(G)''$ be an ideal. We define
$$I(\theta, J) = \{ \Phi \in \ell^{\, 1}(G)'' : \delta_s \Box \Phi = \theta(\delta_s) \Box \Phi \ (s \in G), \; \theta''(\Phi) \in J \}.$$
\end{definition}

\begin{proposition} \label{5.2}
Let $G, \theta$ and $J$ be as in Definition \ref{5.1}. Then $I(\theta, J)$ is an ideal in $\ell^{\, 1}(G)''$.
\end{proposition}

\begin{proof}
Let $\Phi \in I(\theta, J)$, and let $s, t \in G$. Then
\begin{align*}
\delta_s \Box (\delta_t \Box  \Phi) = \delta_{st} \Box \Phi = \theta(\delta_{st}) \Box \Phi = \theta(\delta_s) \Box \theta(\delta_t) \Box \Phi 
=\theta(\delta_s) \Box (\delta_t \Box \Phi).
\end{align*}
By taking linear combinations and weak-* limits, we may conclude that 
$$\delta_s \Box \Psi \Box \Phi = \theta(\delta_s) \Box \Psi \Box \Phi$$ 
for every $\Psi \in \ell^{\, 1}(G)''$ and every $s \in G$.  It is clear that $\delta_s \Box \Phi \Box \Psi = \theta(\delta_s) \Box \Phi \Box \Psi$ for every $\Psi \in \ell^{\, 1}(G)''$. Since $J$ is an ideal and $\theta''(\Phi) \in J$, we  have $\theta''(\Psi \Box \Phi) = \theta''(\Psi) \Box \theta''(\Phi) \in J$ and $\theta''(\Phi \Box \Psi) = \theta''(\Phi) \Box \theta''(\Psi) \in J$ for every $\Psi \in \ell^{\, 1}(G)''$. Finally, we note that $I(\theta, J)$ is clearly a linear space. We have shown that $I(\theta, J)$ is an ideal in $\ell^{\, 1}(G)''$.
\end{proof}

\begin{lemma} \label{5.3}
 Let $G, \theta$ and $J$ be as in Definition \ref{5.1}. Then:
\begin{enumerate}
\item[{\rm (i)}] if $\Phi \in I(\theta, J)$ and $\Psi \in \ell^{\, 1}(G)'',$ then $\Psi \Box \Phi = \theta''(\Psi) \Box \Phi$;
\item[{\rm (ii)}] if $J$ is nilpotent of index $n$, then $I(\theta, J)$ is nilpotent of index at most $n+1$.
\end{enumerate}
\end{lemma}

\begin{proof}
(i) This follows from the identity $\delta_s \Box \Phi = \theta''(\delta_s) \Box \Phi$, and the fact that $\theta''$ is linear and weak-* continuous.

(ii) Given $\Phi_1, \ldots, \Phi_{n+1} \in I(\theta, J),$ we have 
\begin{align*}
\Phi_1 \Box \cdots \Box \Phi_{n+1} &= \theta''(\Phi_1 \Box \cdots \Box \Phi_n) \Box \Phi_{n+1} \\
 &= \theta''(\Phi_1) \Box  \cdots \Box \theta''(\Phi_n) \Box \Phi_{n+1} = 0
\end{align*}
because $\theta''(\Phi_1), \ldots, \theta''(\Phi_n) \in J$. As $\Phi_1, \ldots, \Phi_{n+1}$ were arbitrary, this shows that $I(\theta, J) ^{\Box (n+1)} = \{ 0 \}$.
\end{proof}

The key idea in the proof of Theorem \ref{1.3} is to use invariant means coming from each of the copies of $\Z$ in the direct sum to build more complicated radical elements in $\ell^{\, 1}( \oplus_{i=1}^\infty \Z)''$.  We shall use the following lemma. Recall that, for a group $G$ with subgroups $N$ and $H$, where $N$ is normal in $G$, we say that \textit{$N$ is complemented by $H$} if $H \cap N = \{ e \}$ and $G=HN$. In this case every element of $G$ may be written uniquely as $hn$, for some $h \in H$ and some $n \in N$, and the map $G \rightarrow G$ defined by $hn \mapsto h$ is a group homomorphism.

\begin{lemma} \label{5.4}
Let $G$ be a group with a normal, amenable subgroup $N$ which is complemented by a subgroup $H$. Let $\pi \colon \ell^{\, 1}(G) \rightarrow \ell^{\, 1}(G)$  be the bounded algebra homomorphism defined by $\pi(\delta_{hn}) = \delta_h \ (h \in H, n \in N)$  and let $\iota \colon \ell^{\, 1}(N) \rightarrow \ell^{\, 1}(G)$
 denote the inclusion map. Let $M$ be an invariant mean on $\ell^{\, \infty}(N),$ and write $\widetilde{M} = \iota''(M)$. Then $\widetilde{M}$ satisfies:
\begin{align}
&\delta_s \Box \widetilde{M} = \pi(\delta_s) \Box \widetilde{M} \quad (s \in G); \label{eq5.10} \\
&\pi''(\widetilde{M}) = \delta_e \label{eq5.11}.
\end{align}
\end{lemma}

\begin{proof}
For every  $n \in N$ and every $f \in \ell^{\, 1}(N),$ we have $\delta_n * \iota(f) = \iota(\delta_n*f)$, and so, by taking weak-* limits, we see that $\delta_n \Box \iota''(\Phi) = \iota''(\delta_n \Box \Phi)$ for all $\Phi \in \ell^{\, 1}(N)''$. An arbitrary element $s \in G$ may be written as $s = hn$ for some $h \in H$ and $n \in N$, and so 
\begin{align*}
\delta_s \Box \widetilde{M} &= \delta_h \Box \delta_n \Box \iota''(M) = \delta_h  \Box \iota''(\delta_n \Box M) \\
&= \delta_h \Box \iota''(M) = \pi(\delta_s) \Box \widetilde{M}.
\end{align*}
Hence \eqref{eq5.10} holds.

Define $\varphi_0 \colon \ell^{\, 1}(N) \rightarrow \ell^{\, 1}(N)$ by $\varphi_0 \colon f \mapsto \langle f, 1 \rangle \delta_e \ (f \in \ell^{\, 1}(N))$. It is easily verified that $\pi \circ \iota = \iota \circ \varphi_0$, and so $\pi'' \circ \iota'' = \iota'' \circ \varphi_0''.$ 
We also have $\varphi_0'' (\Phi) = \langle \Phi, 1 \rangle \delta_e  \ (\Phi \in \ell^{\, 1}(N)'')$. Hence 
$$\pi''(\widetilde{M}) = (\pi'' \circ \iota'')(M) = (\iota''\circ \varphi_0'')(M) = \iota''(\langle M, 1 \rangle \delta_e) = \delta_e,$$
establishing \eqref{eq5.11}.
\end{proof}

We have not seen in the literature any instance of a discrete, non-amenable group $G$ for which it is known that $\rad(\ell^{\, 1}(G)'') \neq \{ 0 \}$. However the next corollary gives a large class of easy examples of such groups.

\begin{corollary} \label{5.5}
Let $G$ be a group with an infinite, amenable, complemented, normal subgroup  $N$. Then $\vert \rad (\ell^{\, 1}(G)'') \vert \geq 2^{2^{\vert N \vert}}.$
\end{corollary}

\begin{proof}
Let $M_1, M_2$ be two invariant means in $\ell^{\, 1}(N)''$, and let $\iota$ and $\pi$ be as in Lemma \ref{5.4} . Then by that lemma $\iota''(M_1 - M_2) \in I(\pi,0)$, which is a nilpotent ideal by Lemma \ref{5.3}(ii). The result now follows from the injectivity of $\iota''$ and \cite[Theorem 7.26]{P}.
\end{proof}

We now prove our main theorem.

\begin{proof}[Proof of Theorem \ref{1.3}]
Let $G = \oplus_{i=1}^\infty \Z$, and, given $i \in \N$, write $G_i$ for the $i^{\rm th}$ copy of $\Z$ appearing in this direct sum. 
 Let $\pi_i \colon G \rightarrow G$ be the homomorphism which `deletes' the $i^{\rm \, th}$ coordinate, that is
$$ \pi_i \colon (n_1, n_2, \ldots) \mapsto (n_1, \ldots, n_{i-1}, 0, n_{i+1}, \ldots ).$$
Each map $\pi_i$ gives rise to a bounded homomorphism $\ell^{\, 1}(G) \rightarrow \ell^{\, 1}(G)$, which we also denote by $\pi_i$, given by
$$\pi_i \colon f \mapsto \sum_{s \in G} f(s) \delta_{\pi_i(s)} \quad (f \in \ell^{\, 1}(G)).$$
Similarly, we write $\iota_i \colon G_i \rightarrow G$ for the inclusion map of groups, and $\iota_i \colon \ell^{\, 1}(G_i) \rightarrow \ell^{\, 1}(G)$ for the inclusion of algebras which it induces.

Define a sequence of ideals $I_j$ in $\ell^{\, 1}(G)''$ by $I_1 = I(\pi_1, 0)$ and 
$$I_j = I(\pi_j, I_{j-1}) \quad (j \geq 2).$$
By Lemma \ref{5.3}(ii), each $I_j$ is nilpotent of index at most $j+1$ and the strategy of the proof is to show that the index is exactly $j+1$. 

Fix a free ultrafilter $\mathcal{U}$ on $\N$. Given $i \in \N$ and $n \in \Z$ we write
$$\delta^{(i)}_n = \delta_{(0, \ldots, 0, n, 0, \ldots)},$$
 where $n$ appears in the $i^{\rm \, th}$ place.
 Given $j \in \N$ we define elements $\sigma_j, M_j \in \ell^{\, 1}(G)''$ to be the weak-* limits $\sigma_j = \lim_{k \rightarrow \mathcal{U}} \sigma_{j, k}$ and $M_j = \lim_{k \rightarrow \mathcal{U}} M_{j,k}$, where
$$M_{j,k} = \frac{1}{k} \sum_{i=1}^k \delta_i^{(j)} \quad (j, k \in \N),$$
and
$$\sigma_{j, k} = \frac{1}{k} \sum_{i=1}^k \left( \delta_i^{(j)} - \delta_{-i}^{(j)} \right) \quad (j, k \in \N).$$

We claim that, for each $j \in \N$, $M_j$ and $\sigma_j$ satisfy:
\begin{align}
&\delta_s \Box M_j = \pi_j(\delta_s) \Box M_j  \ (s \in G); \label{eq5.4} \\ 
&\pi_j''(M_j) = \delta_{(0,0, \ldots)}; \label{eq5.5} \\
&\pi_i''(\sigma_j) = \sigma_j \text{ and } \pi_i''(M_j) = M_j \ (i \neq j); \label{eq5.6} \\
&\sigma_j \in I(\pi_j, 0). \label{eq5.7}
\end{align}
Since $ \pi_i''\left(  M_{j,k}  \right) =  M_{j,k} $ and $\pi_i''\left( \sigma_{j, k} \right) = \sigma_{j, k} \ (k \in \N, i\neq j)$, \eqref{eq5.6} follows from the weak-* continuity of $\pi_j''$. We observe that $M_j$ is the image of an invariant mean on $\Z$ under $\iota_j''$, so that we may apply Lemma \ref{5.4} to obtain \eqref{eq5.4} and \eqref{eq5.5}. Similarly, $\sigma_j$ is the image of the difference of two invariant means on $\Z$ under $\iota_j$, so that Lemma \ref{5.4} implies that  $\delta_s \Box \sigma_j = \pi_j(\delta_s) \Box \sigma_j \ (s \in G)$ and $\pi_j'' \left(\sigma_j \right) = 0$, so that \eqref{eq5.7} holds.

We demonstrate that  
\begin{equation} \label{eq5.8}
\sigma_1 \Box \sigma_2 \Box \cdots \Box \sigma_j \neq 0 \quad (j \in \N).
\end{equation}
 To see this, define $h \in \ell^{\,  \infty}(G)$ by
\begin{align*}
 h(n_1, n_2, \ldots) = 
\begin{cases}
1 \qquad \text{if } n_i \geq 0 \text{ for all } i \in \N \\
0 \qquad \text{otherwise.}
\end{cases}
\end{align*}
Clearly, we have
\begin{equation} \label{eq5.13}
\langle \sigma_{i, k}, h \rangle = 1 \quad (i, k \in \N).
\end{equation}
It is easily checked that $\langle \pi_i(\delta_s)*\iota_i(\delta_t), h \rangle = \langle \pi_i(\delta_s), h \rangle \langle \iota_i(\delta_t), h \rangle$ for every $s \in G$ and $t \in G_i$, and it follows from this that 
\begin{equation} \label{eq5.12}
\langle \pi_i(f)*\iota_i(g), h \rangle = \langle \pi_i(f), h \rangle \langle \iota_i(g), h \rangle
\quad (f \in \ell^{\, 1}(G), g \in \ell^{\, 1}(G_i)).
\end{equation}
Given $i, k \in \N$, the element $\sigma_{i,k}$ belongs to the image of $\iota_i$, and, together with \eqref{eq5.6}, \eqref{eq5.12} and \eqref{eq5.13}, this allows us to conclude that, for all $k_1, \ldots, k_j \in \N$, we have
\begin{align*}
\langle \sigma_{1,k_1}*\sigma_{2, k_2}&* \cdots *\sigma_{j, k_j}, h \rangle \\
&= \langle \pi_j(\sigma_{1,k_1}* \cdots *\sigma_{j-1, k_{j-1}})*\sigma_{j, k_j}, h \rangle\\
&=\langle \pi_j(\sigma_{1,k_1}* \cdots *\sigma_{j-1, k_{j-1}}),
 h \rangle \langle \sigma_{j, k_j}, h \rangle \\
&= \langle \pi_{j-1}(\sigma_{1,k_1}* \cdots *\sigma_{j-2, k_{j-2}})*\sigma_{j-1, k_{j-1}},
 h \rangle \langle \sigma_{j, k_j}, h \rangle = \cdots \\
&= \langle \sigma_{1,k_1}, h \rangle \langle \sigma_{2, k_2}, 
h \rangle \cdots \langle \sigma_{j, k_j}, h \rangle = 1.
\end{align*}
Therefore
$$\langle \sigma_1 \Box \sigma_2 \Box \cdots \Box \sigma_j, h \rangle
= \lim_{\underline{k} \rightarrow \mathcal{U}}{}^{(j)} 
\left \langle \sigma_{1,k_1} * \sigma_{2, k_2}* \cdots *\sigma_{j, k_j}, h \right \rangle = 1.$$
Equation \eqref{eq5.8} follows.

We now come to the main argument of the proof. We recursively define $\Lambda_j \in \ell^{\, 1}(G)''$ by $\Lambda_1 = \sigma_1$ and 
$$\Lambda_j = M_j \Box \Lambda_{j-1} + \sigma_j \quad (j \geq 2).$$
 We shall show inductively that each $\Lambda_j$ satisfies:
\begin{align}
&\Lambda_j \in I_j; \label{eq5.1} \\
&\Lambda_j^{\Box j} = \sigma_1 \Box \sigma_2  \Box \cdots \Box \sigma_j; \label{eq5.2} \\
&\pi_i''(\Lambda_j) = \Lambda_j \ (i > j) \label{eq5.3}.
\end{align}
Since by Lemma \ref{5.3}(ii) $I_j^{\Box (j+1)} = \{ 0 \},$ and by \eqref{eq5.8} $\sigma_1 \Box \sigma_2  \Box \cdots \Box \sigma_j \neq 0,$ this will give the result. The base case of the induction holds by \eqref{eq5.7} and \eqref{eq5.6}.

Now assume that the  hypothesis holds up to $j-1$. It follows from \eqref{eq5.7} and \eqref{eq5.4} that $\delta_s \Box \Lambda_j = \pi_j(\delta_s) \Box \Lambda_j \ (s \in G)$. Moreover, by \eqref{eq5.5}, \eqref{eq5.7}, and \eqref{eq5.3} applied to $\Lambda_{j-1}$, we have
\begin{align} \label{eq5.9}
\pi_j''(\Lambda_j) = \pi_j''(M_j) \Box \pi_j''(\Lambda_{j-1})+\pi_j''(\sigma_j) = \Lambda_{j-1},
\end{align}
so that, by the induction hypothesis, $\pi_j''(\Lambda_j) \in I_{j-1}$. Hence \eqref{eq5.1} holds. We see that \eqref{eq5.3} holds for a given $i >j$ because it holds for each of $M_j, \sigma_j$ and $\Lambda_{j-1}$ by \eqref{eq5.6} and the induction hypothesis. Finally, we verify \eqref{eq5.2}:
\begin{align*}
\Lambda_j^{\Box j} &= \pi_j''(\Lambda_j)^{\Box(j-1)} \Box \Lambda_j = \Lambda_{j-1}^{\Box (j-1)} \Box \Lambda_j \\
&= \Lambda_{j-1}^{\Box (j-1)} \Box M_j \Box \Lambda_{j-1} + \Lambda_{j-1}^{\Box (j-1)} \Box \sigma_j  
= \sigma_1 \Box \sigma_2 \Box \cdots \Box \sigma_{j-1} \Box \sigma_j,
\end{align*}
where we have used Lemma \ref{5.3}(i) and \eqref{eq5.9} in the first line, and the fact that $I_{j-1}^{\Box j} = \{ 0 \}$ in the second line to get $\Lambda_{j-1}^{\Box (j-1)} \Box M_j \Box \Lambda_{j-1} =0.$ 

This completes the proof.
\end{proof}

\textit{Remark.} A simpler version of the above argument shows that $\ell^{\, 1}(\Z^2)''$ contains a radical element which is nilpotent of index 3, which is enough to resolve Dales and Lau's question of whether the radical of $L^{\, 1}(G)''$, for $G$ a locally compact group, always has zero square \cite[Chapter 14, Question 3]{DL}. Specifically, this may be achieved by terminating the induction at $j=2$, and otherwise making trivial alterations.

\begin{corollary}
The radical of $\ell^{\, 1}(\oplus_{i=1}^\infty \Z)''$ contains non-nilpotent elements.
\end{corollary}
\begin{proof}
 By a theorem of Grabiner \cite{Gr69}, if every element of $\rad(\ell^{\, 1}(\oplus_{i=1}^\infty \Z)'')$ were nilpotent, then there would be a uniform bound on the index of nilpotency. Hence, by Theorem \ref{1.3}, $\rad(\ell^{\, 1}(\oplus_{i=1}^\infty \Z)'')$ must contain non-nilpotent elements.
\end{proof}

\section{A Weight $\omega$ for Which $\rad(\ell^{\, 1}(\Z, \omega)'')$ is Not Nilpotent}
\noindent
In this section we shall prove Theorem \ref{1.2}. Given a weight $\omega$ on $\Z$ and $r \in \N$, we define 
$\Omega_\omega^{\, (r)} \colon \Z^r \rightarrow (0, 1]$ 
by
$$\Omega_\omega^{\, (r)}(n_1,\ldots, n_r) = \frac{\omega(n_1+n_2+ \cdots +n_r)}{\omega(n_1)\omega(n_2) \cdots \omega(n_r)} \quad (n_1, \ldots, n_r \in \Z)$$
(compare with \cite[Equation 8.7]{DL}). Often we simply write $\Omega^{\, (r)}$ when the weight $\omega$ is clear. As in Section 3 we write $A_\omega = \ell^{\, 1}(\Z, \omega)$.

Our main tool will be Proposition \ref{4.1}. In what follows, the unit ball of a Banach space $E$ is denoted by $B_E$.

\begin{proposition} \label{4.1}	
Let $\omega$ be a weight on $\Z$ and suppose that there is some sequence $(n_k) \subset \Z$ such that 
\begin{equation} \label{eq4.1a}
\liminf_{r \rightarrow \infty} \limsup_{\underline{k} \rightarrow \infty} {}^{(r)} \left[ \Omega^{\,  (r)}(n_{k_1}, \ldots, n_{k_r}) \right]^{1/r}  = 0.
\end{equation}
Let $\Phi$ be a weak-* accumulation point of $\{ \delta_{n_k}/\omega(n_k) : k \in \N \}$. Then $\Phi \in \rad(A_\omega'')\setminus \{ 0 \}.$
\end{proposition}

\begin{proof}
There exists some free filter $\mathcal{U}$ on $\N$ such that 
$$\Phi = \lim_{k \rightarrow \mathcal{U}} \frac{1}{\omega(n_k)} \delta_{n_k},$$
 where the limit is taken in the weak-* topology. Let $\Psi \in B_{A_\omega''}$. Then there exists a net $(a_\alpha)$ in $B_{A_\omega}$ such that, in the weak-* topology, $\lim_{\alpha} a_\alpha = \Psi$.  Let $\lambda \in B_{A_\omega'}$. Then, for each $r \in \N$, we have
\begin{align*}
\vert  \langle (\Psi \Box \Phi)^{\Box r}, \lambda \rangle \vert 
  &= \lim_{\underline{\alpha} \rightarrow \infty, \, \underline{k} \rightarrow \mathcal{U}} {}^{(r)} \left \vert 
	\frac{\langle a_{\alpha_1}* \delta_{n_{k_1}}* \cdots *a_{\alpha_r} * \delta_{n_{k_r}}, \lambda \rangle}
	{\omega(n_{k_1}) \cdots \omega(n_{k_r})}  \right  \vert \\
&= \lim_{\underline{\alpha} \rightarrow \infty, \, \underline{k} \rightarrow \mathcal{U}} {}^{(r)} 
\left \vert \frac{\langle a_{\alpha_1}*  \cdots *a_{\alpha_r} * \delta_{n_{k_1}}* \cdots *\delta_{n_{k_r}}, \lambda \rangle }
{\omega(n_{k_1}) \cdots \omega(n_{k_r})}  \right \vert \\
&\leq \limsup_{\underline{k} \rightarrow \infty} {}^{(r)} \left \Vert \frac{\delta_{n_{k_1}+ \cdots +n_{k_r}}}{\omega(n_{k_1}) \cdots \omega(n_{k_r})} \right \Vert \\
&= \limsup_{\underline{k} \rightarrow \infty} {}^{(r)} \Omega^{\,  (r)}(n_{k_1}, \ldots, n_{k_r}).
\end{align*}
Hence
$$\Vert (\Psi \Box \Phi)^{\Box r} \Vert^{1/r} = \sup_{\lambda \in B_{A_\omega'}} \vert \langle (\Psi \Box \Phi)^{\Box r}, \lambda \rangle \vert ^{1/r} \leq  \limsup_{\underline{k} \rightarrow \infty} {}^{(r)} \left[ \Omega^{\,  (r)}(n_{k_1}, \ldots, n_{k_r}) \right]^{1/r},$$
and so
$\lim_{r \rightarrow \infty} \Vert (\Psi \Box \Phi)^{\Box r} \Vert^{1/r} = 0$ 
by \eqref{eq4.1a}.
Therefore $\Psi \Box \Phi \in \mathcal{Q}(A_\omega'')$. As $\Psi$ was arbitrary, it follows that $\Phi \in \rad(A_\omega'')$. Moreover, $\Phi \neq 0$ because
$$\langle \Phi, \omega \rangle = \lim_{k \rightarrow \mathcal{U}} \left\langle \frac{\delta_{n_k}}{\omega(n_k)}, \omega \right\rangle = 1.$$
This completes the proof.
\end{proof}

\textit{Remark.} In \cite[Example 9.17]{DL} Dales and Lau put forward a candidate for a weight $\omega$ such that $A_\omega''$ is semisimple. They attribute this weight to Feinstein. In the light of Theorem \ref{1.1} this cannot be the case, but in fact this can also be shown directly using Proposition \ref{4.1}, by taking the sequence $(n_k)$ to be the one, also called $(n_k),$ appearing in \cite[Example 9.17]{DL}.

\vskip3mm

Given a (possibly infinite) subset $S \subset \Z$ which generates $\Z$ as a group, we define the \textit{word-length with respect to $S$} of an integer $n$ to be
$$\vert n \vert_S = \min \left \{ r: n= \sum_{i=1}^r \eps_i s_i \text{, for some } s_1, \ldots, s_r \in S, \eps_1, \ldots, \eps_r \in \{ \pm 1 \} \right \}.$$
This is exactly the usual notion of word-length from group theory. Notice that, for any generating set $S$, the function 
$n \mapsto {\rm e}^{\vert n \vert_S}$
defines a weight on $\Z$.

The weight in Theorem \ref{1.2} will be defined as follows. We let  
$$S_0 = \{ 2^{k^2} : k \in \Z^+ \},$$
 set $\eta(n) = \vert n \vert_{S_0}$, and define our weight by $\omega(n) = {\rm e}^{\eta(n)} \ (n \in \Z).$ We also define a sequence of integers $(n_k)$ by 
$$n_k = 2^{k^2} + 2^{(k-1)^2} + \cdots + 1 \quad (k \in \N).$$

\begin{lemma} \label{4.2}	
We have  $\eta(n_k) = k+1 \ (k \in \N)$.
\end{lemma}

\begin{proof}
We proceed by induction on $k \in \N$, the base case being trivial. Take $k>1,$ and assume that the lemma holds for $k-1$. That $\eta(n_k) \leq k+1$ is clear from the definitions, and so it remains to show that $\eta(n_k) \geq k+1$. Observe that, for all $k \in \N$, we have 
\begin{equation} \label{star}
k2^{(k-1)^2} = \frac{k}{2^{2k-1}}2^{k^2} < 2^{k^2}. 
\end{equation}
Assume towards a contradiction that $\eta(n_k) < k+1$. Then we can write 
$n_k = \sum_{i=1}^p c_i 2^{a_i^2},$
for some $p \in \N, c_1, \ldots, c_p \in \Z \setminus \{ 0 \},$ and $a_1, \ldots, a_p \in \Z^+$ such that $\sum_{i=1}^p \vert c_i \vert < k+1$. We may suppose that $a_1 > a_2 > \cdots > a_p$. We first show that $a_1 =k$. If $a_1 \leq k-1$, we find that
$$n_k = \left \vert \sum _{i=1}^p c_i 2^{a_i^2} \right \vert \leq \left( \sum_{i=1}^p \vert c_i \vert \right)2^{(k-1)^2} \leq k2^{(k-1)^2} < n_k$$
by \eqref{star}, a contradiction. Similarly, if $a_1 \geq k+1,$ we have

\begin{align*}
n_k = \left \vert \sum _{i=1}^p c_i 2^{a_i^2} \right \vert &\geq \vert c_1 \vert 2^{{a_1}^2} - \left( \sum_{i=2}^p \vert c_i \vert \right)2^{(a_1-1)^2}
 \geq 2^{{a_1}^2}-(k -1)2^{({a_1}-1)^2} \\
&> a_1 2^{(a_1-1)^2}-(k-1)2^{(a_1-1)^2}\\
&= (a_1 +1 - k)2^{(a_1 - 1)^2} \geq 2 \cdot 2^{k^2} > n_k,
\end{align*}
where we have  used \eqref{star} to obtain the second line. Hence in either case we get a contradiction, so we must have $a_1 = k$, as claimed. 

Observe that $c_1 > 0$, since otherwise
\begin{align*}
\sum _{i=1}^p c_i 2^{a_i^2} \leq -2^{k^2} +\sum_{i=2}^p \vert c_i \vert 2^{a_i^2} \leq -2^{k^2} + (k-1)2^{(k-1)^2} <0. 
\end{align*}
Hence we have deduced that
$$2^{k^2} + n_{k-1} = n_k = 2^{k^2} + (c_1-1)2^{k^2} + \sum_{i=2}^p c_i 2^{a_i^2},$$
which implies that
$$n_{k-1} = (c_1 - 1)2^{k^2} + \sum_{i=2}^p c_i 2^{a_i^2},$$
and this contradicts the induction hypothesis, since $c_1-1 + \sum_{i=2}^p \vert c_i\vert < k$.
\end{proof}

\begin{lemma} \label{4.3}	
Let $j \in \N,$ and set $r = r(j) = 2^{2j+1}$. Then, for all $k_1, \ldots, k_r \geq j,$ we have
$$\left[ \Omega^{\, (r)} (n_{k_1}, \ldots, n_{k_r}) \right ]^{1/r} \leq {\rm e}^{-j}.$$
\end{lemma}

\begin{proof}
First of all, we compute
\begin{align*}
r(j)n_j &= 2^{2j+1} \sum_{i=0}^j 2^{(j-i)^2} = \sum_{i=0}^j 2^{2j+1+j^2-2ij+i^2} \\ 
&= \sum_{i=0}^j 2^{2i} 2^{-2i+2j - 2ij+j^2 +1+i^2}
=\sum_{i=0}^{j} 2^{2i}\cdot 2^{(j+1-i)^2}.
\end{align*}
This implies that
$$\eta(rn_j) \leq \sum_{i=0}^{j} 2^{2i} = \frac{1}{3}(2^{2j} - 1) \leq r,$$
so that, for all $k_1, \ldots, k_r \geq j$, we have
\begin{align*}
\eta(n_{k_1} + \cdots + n_{k_r}) &= \eta[(n_{k_1} - n_j) + \cdots + (n_{k_r} - n_j) +rn_j] \\
&\leq \eta(n_{k_1} - n_j) + \cdots +\eta(n_{k_r} - n_j) +\eta(rn_j) \\
&\leq (k_1 - j)+ \cdots + (k_r-j) + r \\
&= k_1+ \cdots + k_r -(j-1)r.
\end{align*}
Hence, by Lemma \ref{4.2}, we have
\begin{align*}
\left[ \Omega^{(r)}(n_{k_r}, \ldots, n_{k_1}) \right]^{1/r} &\leq \left[ \frac{{\rm e}^{k_1+ \cdots + k_r-(j-1)r}}{{\rm e}^{k_1+1} \cdots {\rm e}^{k_r+1}} \right]^{1/r} \\
&= [{\rm e}^{-(j-1)r-r}]^{1/r} = {\rm e}^{-j},
\end{align*}
as required.
\end{proof}

\begin{lemma} \label{4.4}	
Fix $j \in \N$, and set $r = 2^{2j+1}.$ Let $J \in \N$ satisfy $J \geq j$ and  
\begin{equation*}  
 2^{2k-1}>rk+2r \quad (k \geq J).
\end{equation*}
Then, for all $k_1 \geq k_2 \geq \cdots \geq k_r  \geq J,$ we have
$$\eta(n_{k_1}+ \cdots +n_{k_r}) \geq k_1+k_2+ \cdots +k_r - rJ.$$
\end{lemma}

\begin{proof}
Note that, by our hypothesis on $J$, whenever $k \geq J$ we have
$2^{k^2 - (k-1)^2} > rk+2r$, which implies that
\begin{equation} \label{cond2}
2^{k^2} > rk2^{(k-1)^2}+r2^{(k-1)^2+1} \quad (k \geq J).
\end{equation}
We proceed by induction on $k_1 \geq J$, with the base case corresponding to the case where $k_1 = J$, and hence also $k_2 = \cdots = k_r = J$. Therefore the base hypothesis merely states that $\eta(rn_J) \geq 0$, which is true.

Suppose that $k_1 > J,$ and assume that the lemma holds for all smaller values of $k_1 \geq J$. Assume towards a contradiction that there exist $k_2, \ldots, k_r \in \Z$ such that $k_1 \geq k_2 \geq \cdots \geq k_r \geq J$, and such that
$$\eta(n_{k_1} + \cdots + n_{k_r}) < k_1+ \cdots + k_r - rJ.$$
Then we may write 
$$n_{k_1} + \cdots + n_{k_r} = \sum_{i=1}^p c_i 2^{a_i^2}$$
for some $p \in \N$, some $a_1, \ldots, a_p \in \Z^+,$ and some $c_1, \ldots, c_p \in \Z \setminus \{ 0 \}$,
satisfying $a_1>a_2> \cdots >a_p$ and $\sum_{i=1}^p\vert c_i \vert < k_1+\cdots +k_r -rJ$.
Note that we have
\begin{equation} \label{cond4}
\left \vert \sum_{i=2}^p c_i 2^{a_i^2} \right \vert < (k_1+\cdots +k_r -rJ)2^{a_2^2} \leq k_1r2^{(a_1 - 1)^2}.
\end{equation}

We claim that $a_1 = k_1$. Assume instead that $a_1 \geq k_1+1$. 
Then, using \eqref{cond2} and \eqref{cond4}, we have
\begin{align*} 
\left \vert \sum_{i=1}^p c_i 2^{a_i^2} \right \vert &> 
\vert c_1 \vert 2^{a_1^2} - \left \vert \sum_{i=2}^p c_i2^{a_i^2} \right \vert \\ 
&\geq 2^{a_1^2} - rk_1 2^{(a_1 - 1)^2} 
> r2^{(a_1 -1)^2+1} \\
&\geq r2^{k_1^2+1} \geq n_{k_1}+ \cdots +n_{k_r},
\end{align*}
a contradiction. If, on the other hand, we assume that $a_1 \leq  k_1 - 1$, then \eqref{cond2} implies that
\begin{align*}
\left \vert \sum_{i=1}^p c_i 2^{a_i^2}  \right \vert \leq (k_1+ \cdots + k_r)2^{(k_1-1)^2} \leq rk_1 2^{(k_1-1)^2} < n_{k_1},
\end{align*}
a contradiction. Hence $a_1 = k_1$, as claimed.

Let $d \in \N$ be maximal such that $k_d = k_1$. We claim that $c_1 \geq d$. Firstly, if $c_1$ were negative, we would have
$$\sum_{i=1}^p c_i 2^{a_i^2} \leq -2^{k_1^2} + \left \vert \sum_{i=2}^p c_i 2^{a_i^2} \right \vert 
\leq -2^{k_1^2} + rk_12^{(k_1-1)^2} <0,$$
by \eqref{cond2}, a contradiction. Hence $c_1$ must be positive. Suppose that $c_1 \leq d-1$. Then, using \eqref{cond2} to obtain the second line, we would have
\begin{align*}
\sum_{i=1}^p c_i 2^{a_i^2} &\leq (d-1)2^{k_1^2} +  \left \vert \sum_{i=2}^p c_i 2^{a_i^2} \right \vert 
\leq (d-1)2^{k_1^2} +rk_12^{(k_1-1)^2} \\
&< n_{k_1} + \cdots + n_{k_d} \leq n_{k_1} + \cdots  + n_{k_r},
\end{align*}
again a contradiction. Hence we must  have $c_1 \geq d$, as claimed.

We now complete the  proof. We have
\begin{align*}
c_12^{k_1^2}+\sum_{i=2}^p c_i 2^{a_i^2} &= n_{k_1}+ \cdots +n_{k_r} \\
&= dn_{k_1}+n_{k_{d+1}}+ \cdots +n_{k_r} \\
 &= d2^{k_1^2}+n_{k_1-1}+ \cdots + n_{k_d-1}+n_{k_{d+1}} + \cdots +n_{k_r},
\end{align*}
which implies that
$$(c_1 - d)2^{k_1^2} + \sum_{i=2}^p c_i 2^{a_i^2} = n_{k_1-1}+ \cdots + n_{k_d-1}+n_{k_{d+1}} + \cdots +n_{k_r}.$$
But
$$ (c_1-d) + \sum_{i=2}^p \vert c_i \vert < (k_1 - 1)+ \cdots +(k_d-1)+ k_{d+1}+ \cdots +k_r - Jr,$$
which contradicts the induction hypothesis applied to $k_1-1$.
\end{proof}

\begin{corollary} \label{4.5}	
Fix $j \in \N$, set $r = 2^{2j+1}$, and let $\mathcal{U}$ be an ultrafilter on $\N$. Then $\lim_{\underline{k} \rightarrow \mathcal{U}} \Omega^{(r)}(k_1, \ldots, k_r) >0.$ 
\end{corollary}

\begin{proof}
Let $J$ be as in Lemma \ref{4.4}. Then, for all $k_1, \ldots, k_r \geq J,$ we have
$$\eta(n_{k_1}+ \cdots + n_{k_r}) \geq k_1+ \cdots + k_r -rJ,$$
which, when combined with Lemma \ref{4.2}, implies that
\begin{align*}
\Omega^{(r)}(n_{k_1}, \ldots, n_{k_r}) \geq \frac{{\rm e}^{k_1+ \cdots + k_r - rJ}}{{\rm e}^{k_1+1}\cdots {\rm e}^{k_r+1}} = {\rm e}^{-r(J+1)} > 0.
\end{align*}
This implies the result.
\end{proof}

We can now prove Theorem \ref{1.2}.

\begin{proof}[Proof of Theorem \ref{1.2}]
Let $\Phi$ be a weak-* accumulation point of 
$$\{ \delta_{n_k}/\omega(n_k) : k \in \N \}.$$
Then, by Proposition \ref{4.1} and Lemma \ref{4.3}, $\Phi \in  \rad(A_\omega'')$. 

Take $j \in \N$ and set $r = 2^{2j+1}$. Let $\mathcal{U}$ be a filter on $\N$ such that $\Phi$ is equal to the weak-* limit $\lim_{k \rightarrow \mathcal{U}} \delta_{n_k}/\omega(n_k)$. Then, by Corollary \ref{4.5},
\begin{align*}
\langle \Phi^{\Box r}, \omega \rangle &=\lim_{\underline{k} \rightarrow \mathcal{U}} {}^{(r)} \left \langle \frac{1}{\omega(n_{k_1}) \cdots \omega(n_{k_r})} \delta_{n_{k_1}+ \cdots + n_{k_r}}, \omega \right \rangle \\
&= \lim_{\underline{k} \rightarrow \mathcal{U}} {}^{(r)} \Omega^{(r)}(n_{k_1}, \ldots, n_{k_r}) > 0.
\end{align*}
Hence $\Phi^{\Box r} \neq 0$. Since $r \rightarrow \infty$ as $j \rightarrow \infty$, it follows that $\Phi$ is not nilpotent.
\end{proof}

\section{Open Problems}
\noindent
There are many natural problems that come to mind related to the topic of this paper to which we do not know the answers. 
In the light of Theorem \ref{1.3}, it would be interesting to ask whether or not there exists any locally compact group $G$ for which $\rad(L^{\, 1}(G)'')^{\Box 2} = \{ 0 \}$. We do not know of such a group, and we are unable to say whether or not this happens for $G= \Z$.

The question of whether the second dual of a Beurling algebra can ever be semisimple remains open. Moreover, even when the weight is trivial, the fact that this cannot occur has only been established for non-discrete groups, and discrete amenable groups. It would seem that determining whether or not, for example, $\rad(\ell^{\, 1}(F_2)'') \neq \{ 0 \},$ where $F_2$ denotes the free group on two generators, requires some totally new ideas.

Another natural setting for this type of question is the Fourier algebra of a locally compact group  $A(G)$, and one could ask whether we always have $\rad(A(G)'') \neq \{ 0 \}$. Questions about Arens multiplication in $A(G)''$ are actively studied: see, for example, \cite{Los, FNS}. When $G$ is abelian $A(G) \cong L^{\, 1}(\widehat{G})$, so that in this case $\rad(A(G)'') \neq \{ 0 \}$ by the known results mentioned above.

\subsection*{Acknowledgements}
\noindent
I would like to thank my doctoral supervisors Garth Dales and Niels Laustsen for their invaluable feedback on earlier versions of the present work. This research was conducted whilst I was studying for a PhD funded by the EPSRC. I would also like to thank the referee for his/her nice suggestion that a section discussing open problems be included.

\end{document}